\documentclass[a4paper]{article}

\usepackage{graphicx} 
\usepackage{geometry}

\usepackage[section]{placeins}

\usepackage{amssymb} 
\usepackage[intlimits]{amsmath}
\usepackage{amsfonts}
\usepackage{amsthm} 
     
\usepackage{mathptmx}
\allowdisplaybreaks
\usepackage{cite}
\usepackage{hyperref}
\hypersetup{colorlinks=true,linkcolor=black,citecolor=black}

\usepackage{ulem}
\usepackage{color}
\usepackage[dvipsnames]{xcolor}

\newcommand{\be}{\begin{equation}}
\newcommand{\ee}{\end{equation}}

\newcommand{\R}{\mathbb{R}} 
\newcommand{\N}{\mathbb{N}} 

\newcommand{\dist}{\textnormal{dist}} 
\newcommand{\diam}{\textnormal{diam}} 
\newcommand{\essinf}{\textnormal{inf}} 

\newcommand{\Ds}{(-\Delta)^s}

\newcommand{\cE}{{\mathcal E}}

\newcommand{\cH}{{\mathcal H}}

\newcommand{\cV}{{\mathcal V}}

\renewcommand{\O }{\Omega }

\newcommand{\ov}{\overline}

\newtheorem{theorem}{Theorem}[section]
\newtheorem{lemma}[theorem]{Lemma}
\newtheorem{proposition}[theorem]{Proposition}

\theoremstyle{remark}
\newtheorem{remark}[theorem]{Remark}

\theoremstyle{definition}

\numberwithin{equation}{section}

\title{Symmetry of odd solutions to equations with fractional Laplacian}

\author{
	Sidy M. Djitte\footnote{\textit{Adress:} Institut f\"ur Mathematik, Goethe-Universit\"at Frankfurt a.M., Robert-Mayer-Str. 10, 60325 Frankfurt a.M., Germany and African Institute for Mathematical Sciences in Senegal (AIMS Senegal), KM 2, Route de Joal, B.P. 14 18. Mbour, Senegal.\newline \textit{Email addresses:\ }\ \texttt{djitte@math.uni-frankfurt.de, sidy.m.djitte@aims-senegal.org}}
	 \ and \ 
	Sven Jarohs\footnote{\textit{Adress:} Institut f\"ur Mathematik, Goethe-Universit\"at Frankfurt a.M., Robert-Mayer-Str. 10, 60325 Frankfurt a.M., Germany.\newline \textit{Email address:\ }\ \texttt{jarohs@math.uni-frankfurt.de}.}
}

\date{\today}

\begin{document}
\maketitle

\pdfbookmark[1]{Abstract}{Abstract}
\begin{abstract}
We present a symmetry result to solutions of equations involving the fractional Laplacian in a domain with at least two perpendicular symmetries. We show that if the solution is continuous, bounded, and odd in one direction such that it has a fixed sign on one side, then it will be symmetric in the perpendicular direction. Moreover, the solution will be monotonic in the part where it is of fixed sign. In addition, we present also a class of examples in which our result can be applied. 
\end{abstract}

\bigskip

\setcounter{section}{1}

In the following, we study symmetries of odd solutions to the nonlinear problem
\begin{equation}\label{main-prob}
\left\{\begin{aligned} (-\Delta)^su&=f(x,u)&&\text{in $\Omega$}\\
u&=0 &&\text{in $\R^N\setminus \Omega$}\end{aligned}\right.
\end{equation}
where $\Omega\subset \R^N$ is an open set, $f\in C(\Omega\times \R)$, and $(-\Delta)^s$, $s\in(0,1)$ is the fractional Laplacian given for $\phi\in C^{\infty}_c(\R^N)$ by
\[
(-\Delta)^s\phi(x)=\frac{c_{N,s}}{2}\int_{\R^N}\frac{2\phi(x)-\phi(x+y)-\phi(x-y)}{|y|^{N+2s}}\ dy,
\]
with a normalization constant $c_{N,s}>0$. 

Symmetries of nonnegative solutions to problem \eqref{main-prob} have been studied in detail by various authors (see \cite{BLW05,CFY13,FW13-2,JW14,JW16}), where $f$ satisfies some monotonicity and symmetry in $x_1$ and $\Omega$ is symmetric in $x_1$. Here, we aim at investigating \eqref{main-prob}, where $\Omega$ has two perpendicular symmetries and the solution $u$ is odd in one of these directions. To give a precise framework of our statements, we assume the following:

\begin{enumerate}
\item[(D)] $\Omega\subset \R^N$ with $N\in \N$, $N\geq 2$ is open and bounded and, moreover, convex and symmetric in the directions $x_1$ and $x_N$. That is, for every $(x_1,\ldots,x_N)\in \Omega$, $t,\tau\in[-1,1]$ we have $(tx_1,x_2\ldots,x_{N-1},\tau x_N)\in \Omega$.
\item[(F1)] $f\in C(\Omega\times \R)$ and for every bounded set $K\subset \R$ there is $L=L(K)>0$ such that
\[
\sup_{x\in \Omega} |f(x,u)-f(x,v)|\leq L|u-v|\quad\text{for all $u,v\in K$.}
\]
\item[(F2)] $f$ is symmetric in $x_1$ and monotone in $|x_1|$. That is, for every $u\in \R$, $x\in \Omega$, and $t\in[-1,1]$ we have $f(tx_1,x_2,\ldots,x_N,u)\geq f(x,u)$.
\end{enumerate}

In this work, we consider \textit{weak} solutions of \eqref{main-prob}, i.e., $u\in \cH^s_0(\Omega):=\{v\in H^s(\R^N)\;:\; u=0 \text{ on $\R^{N}\setminus \Omega$}\}$ is called a (weak) solution of \eqref{main-prob}, if
$$
\cE_s(u,v)=\int_{\Omega} f(x,u(x)) v(x)\ dx\quad\text{for all $v\in \cH^s_0(\Omega)$,}
$$
whenever the right-hand side exists, where
\begin{equation}\label{bilinearform}
\cE_s(u,v)=\frac{c_{N,s}}{2}\int_{\R^N}\int_{\R^N}\frac{(u(x)-u(y))(v(x)-v(y))}{|x-y|^{N+2s}}\ dxdy
\end{equation}
is the bilinearform associated to $(-\Delta)^s$. Here, $H^s(\R^N)=\{u\in L^2(\R^N)\;:\; \cE_s(u,u)<\infty\}$ is the usual fractional Hilbert space of order $s$ (see e.g. \cite{DPV,BV15}).

Denote $e_i:=(\delta_{ij})_{1\leq j\leq N}\in \R^N$, where $\delta_{ij}=1$ if $j=i$ and $0$ otherwise is the usual Kronecker Delta. Moreover, for $\lambda\in\R$, consider the halfspace 
\begin{equation}\label{halfspace}
H_{i,\lambda}:=\{x\in \R^N\;:\; x\cdot e_i>\lambda\}=\{x\in \R^N\;:\; x_i>\lambda\}
\end{equation} 
and denote by 
\begin{equation}\label{reflection}
r_{i,\lambda}:\R^N\to\R^N, \quad r_{i,\lambda}(x)=2(\lambda-x\cdot e_i)e_i+x
\end{equation}
the reflection of $x$ at $\partial H_{i,\lambda}(\lambda)$. Note that $r_{1,0}(\Omega)=r_{N,0}(\Omega)=\Omega$, if assumption (D) is satisfied.

We call $u:\R^N\to\R$ \textit{symmetric with respect to $H_{i,\lambda}$}, if $u\circ r_{i,\lambda}=u$ and we call $u$ \textit{antisymmetric with respect to $H_{i,\lambda}$}, if $u\circ r_{i,\lambda}=-u$.

\begin{theorem}\label{main-thm1}
Let $\Omega\subset \R^N$ satisfy (D), $f\in C(\Omega\times\R)$ satisfy (F1) and (F2), and let $u\in \cH^s_0(\Omega)$ be a continuous bounded solution of \eqref{main-prob}, which is antisymmetric with respect to $H_{N,0}$ and $u\geq 0$ in $H_{N,0}\cap \Omega$. Then $u$ is symmetric with respect to $H_{1,0}$ and either $u\equiv 0$ in $\Omega$ or $u|_{\Omega\cap H_{1,0}\cap H_{N,0}}$ is strictly decreasing in $x_1$, that is, for every $x,y\in \Omega\cap H_{1,0}\cap H_{N,0}$ with $x_1<y_1$ we have $u(x)>u(y)$.
\end{theorem}

We note that Theorem \ref{main-thm1} is not surprising in the local case, where $(-\Delta)^s$ is considered with $s=1$, if we have $u>0$ in $H_{N,0}\cap \Omega$. In this case, the conclusion follows by simply considering the solution restricted to its part of nonnegativity and apply the usual symmetry result due to \cite{GNN79}. We emphasize however, that if this positivity assumption is reduced to a nonnegativity assumption, then in general the claimed monotonicity is not true in the local case and presents a purely nonlocal feature. Moreover, in the nonlocal setting, we are not able to simply restrict the solution to its set of nonnegativity. Due to this, we present in Section \ref{linear} below new maximum principles for \textit{doubly antisymmetric} functions to certain linear problems, which we believe are of independent interest.\\

Let us emphasize that if $u\in L^{\infty}(\R^N)$ is antisymmetric with respect to $H_{N,0}$, it follows that for any $x\in H_{N,0}$, such that $u$ is regular enough at $x$, we have with a change of variables
\[
(-\Delta)^su(x)=c_{N,s}\lim_{\epsilon\to0^+}\int_{H_{N,0}\setminus B_{\epsilon}(x)}(u(x)-u(y))\Big(\frac{1}{|x-y|^{N+2s}}-\frac{1}{|x-r_{N,0}(y)|^{N+2s}}\Big)\ dy=(-\Delta|_{H_{N,0}})^su(x),
\]
where $(-\Delta|_{H_{N,0}})^s$ denotes the so-called spectral fractional Laplacian (c.f. \cite{CT10} for $s=1/2$). In particular, this \textit{difference of the kernel function} does not meet the assumptions needed to conclude the symmetry result by a restriction to $H_{N,0}$ and applying statements of \cite{JW16}.

In the particular case, where $\Omega=B_1(0)$ is the unitary ball, it was shown in \cite{FFTW20} that the second eigenfunction of the fractional Laplacian in $B_1(0)$, denoted by $\phi_2$, is odd and can be chosen to be positive in $\{x_N>0\}$. Due to the regularity of $\phi_2$, Theorem \ref{main-thm1} yields that for $i=1,\ldots,N-1$ we have
\begin{enumerate}
\item $\phi_2$ is symmetric with respect to $H_{i,0}$ (see also \cite{J16}) and 
\item $\phi_2|_{\{x_1>0\}}$ is decreasing in $x_i>0$.
\end{enumerate}
We emphasize that such a statement already follows due to \cite{FFTW20} combined with \cite{D-al}, since thus the second eigenfunction can be written as a product of the first eigenfunction with a homogeneous function.
\medskip

To give a more generalized application of our results to a class of nonlinear problems, we consider for $1<p<\frac{2N}{N-2s}$ the minimization problem
\be\label{eq-1.5}
\lambda_{1,p}(\Omega):=\min_{\substack{u\in \cH^s_0(\Omega)\\ u\neq 0}}\frac{\cE_s(u,u)}{\Big(\int_{\Omega}|u(x)|^p\ dx\Big)^{2/p}}.
\ee
Clearly, the minimizer exists and is a solution of \eqref{main-prob} with $f(x,u)=|u|^{p-2}u$ (see e.g. \cite{DPV,SV12,SV13}) and, since $\cE_s(|u|,|u|)\leq \cE_s(u,u)$ it can be chosen to be positive. For more information about the minimization problem \eqref{eq-1.5} we refer to \cite{FE}. In the local case $s=1$, this is a well known problem, see e.g. \cite{FL, BK}. If $\Omega$ satisfies (D), then it follows that this minimizer is also symmetric with respect to the symmetries of $\Omega$ (see \cite{JW14}). In this case, we can also consider the minimizer in the set of $\cH^s_0(\Omega)$-functions, which satisfy $u=-u\circ r_{N,0}$, that is
\be\label{eq-1.6}
\lambda_{1,p}^-(\Omega):=\min_{\substack{u\in \cH^s_0(\Omega)\\ u\neq 0\\ u=-u\circ r_{N,0}}}\frac{\cE_s(u,u)}{\Big(\int_{\Omega}|u(x)|^p\ dx\Big)^{2/p}}.
\ee
In the next theorem, we prove that minimizers of \eqref{eq-1.6} have constant sign in $\Omega\cap H_{N,0}$ and in the particular case $p=2$ we also prove a simplicity result for $\lambda_{1,p}^-(\Omega)$.

\begin{theorem}\label{main-thm2}
Let $1<p<\frac{2N}{N-2s}$ with $N\geq 2$ and let $\Omega\subset \R^N$ satisfy (D) with $\partial\Omega$ of class $C^{1,1}$. Then there is a nontrivial solution $u\in \cH^s_0(\Omega)$ of 
\begin{equation}\label{main-prob2}
\left\{\begin{aligned} (-\Delta)^su&=\lambda_{1,p}^-(\Omega)|u|^{p-2}u&&\text{in $\Omega$}\\
u&=0 &&\text{in $\R^N\setminus \Omega$}\end{aligned}\right.
\end{equation}
 which is continuous, bounded, and antisymmetric with respect to $H_{N,0}$. Moreover, $u$ is of one sign in $\Omega\cap H_{N,0}$ and hence it is symmetric with respect to $H_{1,0}$ and $u|_{\Omega\cap H_{1,0}\cap H_{N,0}}$ is strictly decreasing in $x_1$. In particular, $u$ can be chosen to be positive in $H_N\cap\Omega$. Furthermore, if $p=2$, then the minimizer is unique up to a sign.
\end{theorem}

The existence, as mentioned above, follows immediately from a minimization problem. Moreover, by the known regularity theory it follows that indeed we have $u\in C^{\infty}(\Omega)\cap C^s(\R^N)$, see e.g. \cite{RS14,G15:2}. We show here that this minimizer can actually be chosen to be nonnegative in $\Omega\cap H_N$ and thus the conclusion follows from Theorem \ref{main-thm1}.\\

This work is organized as follows. In Section \ref{linear} we give the framework for supersolutions and maximum principles used later on. Section \ref{symmetry} is devoted to prove Theorem \ref{main-thm1} and in Section \ref{application} we show Theorem \ref{main-thm2}.

\paragraph{Notation} The following notation is used. For subsets $D,U\subset \R^N$ we write $\dist(D,U):=\inf\{|x-y| \;:\; x\in D,\ y\in U\}$.
If $D =\{x\}$ is a singleton, we write $\dist(x,U)$ in place of $\dist({x},U)$. For $U\subset \R^N$ and $ r> 0$
we consider $B_r(U) :=\{x \in \R^N\;:\; \dist(x,U) < r\}$, and we let, as usual $B_r(x) = B_r (\{x\})$ be the
open ball in $\R^N$ centered at $x\in \R^N$ with radius $r > 0$. For any subset $M \subset \R^N$, we denote
by $1_M : \R^N\to\R$ the characteristic function of $M$ and by $\diam(M)$ the diameter of $M$. If $M$ is
measurable, $|M|$ denotes the Lebesgue measure of $M$. Moreover, if $w : M \to \R$ is a function,
we let $w^+= \max\{w,0\}$ resp. $w^- = -\min\{w,0\}$ denote the positive and negative part of w,
respectively, so that $w=w^+-w^-$. Finally, $H_{i,\lambda}$ and $r_{i,\lambda}$ are as defined in \eqref{halfspace} and resp. \eqref{reflection} for $i\in\{1,\ldots,N\}$ and $\lambda\in \R$. Finally, $\Omega\subset \R^N$ is always an open set satisfying (D).

\section{A linear problem} \label{linear}

For this section, we fix $\lambda,\mu\in \R$ and denote $H_1:=H_{1,\mu}$ and $H_2:=H_{N,\lambda}$. Similarly, $r_1:=r_{1,\mu}$ and $r_2:=r_{N,\lambda}$. We call $w:\R^N\to\R$ \textit{doubly antisymmetric (with respect to $H_{1}$ and $H_{2}$)}, if
\[
w\circ r_i=-w\quad\text{in $\R^N$ for $i=1,2$.}
\]
Moreover, if $U\subset \R^N$ is open, we let 
\[
\cV^s(U)=\Big\{u\in L^2_{loc}(\R^N)\;:\; \rho_s(w,U):=\int_{U}\int_{\R^N}\frac{(w(x)-w(y))^2}{|x-y|^{N+2s}}\ dxdy<\infty\Big\}.
\]
Note that clearly $\cH^s_0(U)\subset H^s(\R^N)\subset \cV^s(\R^N)\subset \cV^s(U)$. In the following Lemma we collect some statements corresponding to the space $\cV^s(U)$. The proofs can be found e.g. in \cite{JW16,JW19,JW20}.

\begin{lemma}\label{properties vs}
Let $U\subset \R^N$ open and bounded. Then the following hold.
\begin{enumerate}
\item $\cE_s$ is well defined on $\cV^s(U)\times \cH^s_0(U)$. 
\item If $w\in \cV^s(U)$, then also $w^{\pm},|w|\in \cV^s(U)$. Moreover, if $w\geq 0$ in $\R^N\setminus U$, then $w^-\in\cH^s_0(U)$ and we have 
$$
\cE_s(w^-,w^-)\leq -\cE_s(w,w^-).
$$
\item\label{item3} Let $i=1$ or $i=2$ and $U\subset H_i$. If $w\in \cV^s(U)$ is antisymmetric in $x_i$, then $w1_{H_i}\in \cV^s(U)$. Moreover, if $w\geq 0$ in $H_i\setminus U$, then $w^{-}1_{H_i}\in \cH^s_0(U)$ and $\cE_s(w^-1_{H_i},w^-1_{H_i})\leq -\cE_s(w,w^-1_{H_i})$.
\end{enumerate}
\end{lemma}

The following Lemma gives an extension of Lemma \ref{properties vs}.\ref{item3} to the case of doubly antisymmetric functions.

\begin{lemma}\label{test-function}
Let $U\subset H_1\cap H_2$ and $w\in \cV^s(U_{1,2})$ be doubly antisymmetric, where $U_{1,2}=U\cup r_1(U)\cup r_2(U)\cup r_1(r_2(U))$. Then $v=w^-1_{H_1}1_{H_2}-w^+1_{H_1^c}1_{H_2}\in \cH^s_0(U\cup r_1(U))$ and we have
\[
\cE_s(w,v)+\cE_s(v ,v)\leq0,
\]
where equality can only hold if $v\equiv 0$, that is, if $w\geq0$ in $H_1\cap H_2$.
\end{lemma}

\begin{proof}
First note that since $w$ is antisymmetric with respect to $H_i$, $i=1,2$, Lemma \ref{properties vs} and its proof imply $w_i:=w1_{H_i}\in \cV^s(U\cup r_j(U))$, $i,j=1,2$, $i\neq j$ and
\[
\rho_s(w_i,U\cup r_j(U))\leq \rho_s(w,U_{1,2})\quad\text{and}\quad \cE_s(w_i^-,w_i^-)\leq -\cE_s(w,w_i^-)\quad\text{for $i,j=1,2$, $i\neq j$.}
\]
Similarly, also $w_{2}$ is antisymmetric with respect to $H_1$ (resp. $w_1$ with respect to $H_2$) and thus also $w_{1,2}:=w_11_{H_2}=w_21_{H_1}\in \cV^s(U)$ with 
\[
\rho_s(w_{1,2},U)\leq \min\Big\{\rho_s(w_1,U\cup r_2(U)),\rho_s(w_2,U\cup r_1(U))\Big\}
\]
and it holds
\[
\cE_s(w_{1,2}^-,w_{1,2}^-)\leq -\max\Big\{\cE_s(w_1,w_{1,2}^-),\cE_s(w_2,w_{1,2}^-)\Big\}.
\]
Similarly, we also have $w_{r_1,2}=w_21_{H_1^c}\in \cV^s(r_1(U))$ with 
\[
\rho_s(w_{r_1,2},r_1(U))\leq \rho_s(w_2,U\cup r_1(U)).
\]
It thus follows that $v=w^-1_{H_1}1_{H_2}-w^+1_{H_1^c}1_{H_2}=w_{1,2}^--w_{r_1,2}^+\in \cH^s_0(U\cup r_1(U))$.
Using the monotonicity of $|\cdot|$ and the antisymmetry of $w$ and denoting $r_{1,2}:=r_1\circ r_2=r_2\circ r_1$ we have by several rearrangements and substitutions
\begin{align}
&\frac{2}{c_{N,s}}\Big(\cE_s(w,v)+\cE_s(v,v)\Big)=\int_{H_2}\int_{H_2}\frac{[(w+v)(x)-(w+v)(y)][v(x)-v(y)]}{|x-y|^{N+2s}}\ dxdy+\int_{H_2^c}\int_{H_2^c}\ldots+2\int_{H_2}\int_{H_2^c}\ldots\notag\\
&=\int_{H_2}\int_{H_2}\frac{[(w+v)(x)-(w+v)(y)][v(x)-v(y)]}{|x-y|^{N+2s}}\ dxdy-2\int_{H_2}\int_{H_2}\frac{[ w (r_2(x))-(w+v)(y)]v(y)}{|r_2(x)-y|^{N+2s}}\ dxdy\notag\\
&=\int_{H_2}\int_{H_2}\frac{[(w+v)(x)-(w+v)(y)][v(x)-v(y)]}{|x-y|^{N+2s}}\ dxdy-2\int_{H_2}\int_{H_2}\frac{[ -w (x)-(w+v)(y)]v(y)}{|r_2(x)-y|^{N+2s}}\ dxdy\notag\\
&=\int_{H_1\cap H_2}\int_{H_1\cap H_2}\frac{[(w+v)(x)-(w+v)(y)][v(x)-v(y)]}{|x-y|^{N+2s}}\ dxdy+\int_{H_2\setminus H_1}\int_{H_2\setminus H_1}\ldots +2\int_{H_2\cap H_1}\int_{H_2\setminus H_1}\ldots \notag\\
&\quad -2\int_{H_2\cap H_1}\int_{H_2\cap H_1}\frac{[ -w (x)-(w+v)(y)]v(y)}{|r_2(x)-y|^{N+2s}}\ dxdy- 2\int_{H_2\setminus H_1}\int_{H_2\setminus H_1}\ldots-2\int_{H_2\cap H_1}\int_{H_2\setminus H_1}\ldots-2\int_{H_2\setminus H_1}\int_{H_2\cap H_1}\ldots\notag\\
&=\int_{H_1\cap H_2}\int_{H_1\cap H_2}\frac{[(w+w^-)(x)-(w+w^-)(y)][w^-(x)-w^-(y)]}{|x-y|^{N+2s}}\ dxdy\notag\\
&\quad -\int_{H_2\setminus H_1}\int_{H_2\setminus H_1}\frac{[(w-w^+)(x)-(w-w^+)(y)][w^+(x)-w^+(y)]}{|x-y|^{N+2s}}\ dxdy\notag\\
&\quad  -2\int_{H_2\cap H_1}\int_{H_2\setminus H_1}\frac{[(w-w^+)(x)-(w+w^-)(y)][w^+(x)+w^-(y)]}{|x-y|^{N+2s}}\ dxdy \notag\\
&\quad -2\int_{H_2\cap H_1}\int_{H_2\cap H_1}\frac{[ -w (x)-(w+w^-)(y)]w^-(y)}{|r_2(x)-y|^{N+2s}}\ dxdy + 2\int_{H_2\setminus H_1}\int_{H_2\setminus H_1}\frac{[ -w (x)-(w-w^+)(y)]w^+(y)}{|r_2(x)-y|^{N+2s}}\ dxdy\notag\\
&\quad +2\int_{H_2\setminus H_1}\int_{H_2\cap H_1}\frac{[ -w (x)-(w-w^+)(y)]w^+(y)}{|r_2(x)-y|^{N+2s}}\ dxdy-2\int_{H_2\cap H_1}\int_{H_2\setminus H_1}\frac{[ -w (x)-(w+w^-)(y)]w^-(y)}{|r_2(x)-y|^{N+2s}}\ dxdy\notag\\
&=\int_{H_1\cap H_2}\int_{H_1\cap H_2}\frac{[w^+(x)-w^+(y)][w^-(x)-w^-(y)]}{|x-y|^{N+2s}}\ dxdy\notag\\
&\quad -\int_{H_2\setminus H_1}\int_{H_2\setminus H_1}\frac{[-w^-(x)+w^-(y)][w^+(x)-w^+(y)]}{|x-y|^{N+2s}}\ dxdy\notag\\
&\quad  -2\int_{H_2\cap H_1}\int_{H_2\setminus H_1}\frac{[-w^-(x)-w^+(y)][w^+(x)+w^-(y)]}{|x-y|^{N+2s}}\ dxdy \notag\\
&\quad -2\int_{H_2\cap H_1}\int_{H_2\cap H_1}\frac{[ -w (x)-w^+(y)]w^-(y)}{|r_2(x)-y|^{N+2s}}\ dxdy + 2\int_{H_2\setminus H_1}\int_{H_2\setminus H_1}\frac{[ -w (x)+w^-(y)]w^+(y)}{|r_2(x)-y|^{N+2s}}\ dxdy\notag\\
&\quad +2\int_{H_2\setminus H_1}\int_{H_2\cap H_1}\frac{[ -w (x)+w^-(y)]w^+(y)}{|r_2(x)-y|^{N+2s}}\ dxdy-2\int_{H_2\cap H_1}\int_{H_2\setminus H_1}\frac{[ -w (x)-w^+(y)]w^-(y)}{|r_2(x)-y|^{N+2s}}\ dxdy\notag\\
&=-\int_{H_1\cap H_2}\int_{H_1\cap H_2}\frac{w^+(x)w^-(y)+w^+(y)w^-(x)}{|x-y|^{N+2s}}\ dxdy -\int_{H_2\setminus H_1}\int_{H_2\setminus H_1}\frac{w^-(x)w^+(y)+w^-(y)w^+(x) }{|x-y|^{N+2s}}\ dxdy\notag\\
&\quad  +2\int_{H_2\cap H_1}\int_{H_2\setminus H_1}\frac{w^-(x)w^-(y)+w^+(y)w^+(x)}{|x-y|^{N+2s}}\ dxdy \notag\\
&\quad +2\int_{H_2\cap H_1}\int_{H_2\cap H_1}\frac{ w (x) w^-(y)}{|r_2(x)-y|^{N+2s}}\ dxdy - 2\int_{H_2\setminus H_1}\int_{H_2\setminus H_1}\frac{ w (x) w^+(y)}{|r_2(x)-y|^{N+2s}}\ dxdy\notag\\
&\quad -2\int_{H_2\setminus H_1}\int_{H_2\cap H_1}\frac{ w (x) w^+(y)}{|r_2(x)-y|^{N+2s}}\ dxdy+2\int_{H_2\cap H_1}\int_{H_2\setminus H_1}\frac{ w (x) w^-(y)}{|r_2(x)-y|^{N+2s}}\ dxdy\notag\\
&= -2\int_{H_1\cap H_2}\int_{H_1\cap H_2}\frac{w^+(x)w^-(y) }{|x-y|^{N+2s}}\ dxdy -2\int_{H_2\setminus H_1}\int_{H_2\setminus H_1}\frac{w^-(x)w^+(y)  }{|x-y|^{N+2s}}\ dxdy\notag\\
&\quad  +2\int_{H_2\cap H_1}\int_{H_2\setminus H_1}\frac{w^-(x)w^-(y)+w^+(y)w^+(x)}{|x-y|^{N+2s}}\ dxdy \notag\\
&\quad +2\int_{H_2\cap H_1}\int_{H_2\cap H_1}\frac{ w^+(x) w^-(y)-w^-(x) w^-(y)}{|r_2(x)-y|^{N+2s}}\ dxdy - 2\int_{H_2\setminus H_1}\int_{H_2\setminus H_1}\frac{ w^+ (x) w^+(y)-w^-(x)w^+(y)}{|r_2(x)-y|^{N+2s}}\ dxdy\notag\\
&\quad -2\int_{H_2\setminus H_1}\int_{H_2\cap H_1}\frac{ w^+(x) w^+(y)-w^-(x)w^+(y)}{|r_2(x)-y|^{N+2s}}\ dxdy+2\int_{H_2\cap H_1}\int_{H_2\setminus H_1}\frac{ w^+(x) w^-(y)-w^-(x)w^-(y)}{|r_2(x)-y|^{N+2s}}\ dxdy\notag\\
&= -2\int_{H_2\cap H_1}\int_{H_2\cap H_1}\frac{w^-(x) w^-(y)}{|r_2(x)-y|^{N+2s}}\ dxdy- 2\int_{H_2\setminus H_1}\int_{H_2\setminus H_1}\frac{ w^+ (x) w^+(y)}{|r_2(x)-y|^{N+2s}}\ dxdy\notag\\
&\quad -2\int_{H_2\cap H_1}\int_{H_2\cap H_1}w^+(x)w^-(y)\Big(\frac{1}{|x-y|^{N+2s}}-\frac{1}{|r_2(x)-y|^{N+2s}}\Big)\ dxdy\notag\\
&\quad- 2\int_{H_2\setminus H_1}\int_{H_2\setminus H_1}w^-(x)w^+(y)\Big(\frac{1}{|x-y|^{N+2s}}-\frac{1}{|r_2(x)-y|^{N+2s}}\Big)\ dxdy\notag\\
&\quad +2\int_{H_2\cap H_1}\int_{H_2\cap H_1}\frac{w^-(r_1(x))w^-(y)+w^+(y)w^+(r_1(x))}{|r_1(x)-y|^{N+2s}}\ dxdy\notag\\
&\quad -2\int_{H_2\cap H_1}\int_{H_2\cap H_1}\frac{ w^+(x) w^+(r_1(y))-w^-(x)w^+(r_1(y))}{|r_2(x)-r_1(y)|^{N+2s}}\ dxdy\notag\\
&\quad +2\int_{H_2\cap H_1}\int_{H_2\cap  H_1}\frac{ w^+(r_1(x)) w^-(y)-w^-(r_1(x))w^-(y)}{|r_{1,2}(x)-y|^{N+2s}}\ dxdy\notag\\
&= -2\int_{H_2\cap H_1}\int_{H_2\cap H_1}\frac{w^-(x) w^-(y)}{|r_2(x)-y|^{N+2s}}\ dxdy - 2\int_{H_2\setminus H_1}\int_{H_2\setminus H_1}\frac{ w^+ (x) w^+(y)}{|r_2(x)-y|^{N+2s}}\ dxdy\notag\\
&\quad -2\int_{H_2\cap H_1}\int_{H_2\cap H_1}w^+(x)w^-(y)\Big(\frac{1}{|x-y|^{N+2s}}-\frac{1}{|r_2(x)-y|^{N+2s}}\Big)\ dxdy\notag\\
&\quad- 2\int_{H_2\setminus H_1}\int_{H_2\setminus H_1}w^-(x)w^+(y)\Big(\frac{1}{|x-y|^{N+2s}}-\frac{1}{|r_2(x)-y|^{N+2s}}\Big)\ dxdy\notag\\
&\quad +2\int_{H_2\cap H_1}\int_{H_2\cap H_1}\frac{w^+(x)w^-(y)+w^+(y)w^-(x)}{|r_1(x)-y|^{N+2s}}\ dxdy\notag\\
&\quad -2\int_{H_2\cap H_1}\int_{H_2\cap H_1}\frac{ w^+(x) w^-(y)-w^-(x)w^-(y)}{|r_2(x)-r_1(y)|^{N+2s}}\ dxdy+2\int_{H_2\cap H_1}\int_{H_2\cap  H_1}\frac{ w^-(x) w^-(y)-w^+(x)w^-(y)}{|r_{1,2}(x)-y|^{N+2s}}\ dxdy\notag\\
&= -2\int_{H_2\cap H_1}\int_{H_2\cap H_1}\frac{w^-(x) w^-(y)}{|r_2(x)-y|^{N+2s}}\ dxdy- 2\int_{H_2\cap H_1}\int_{H_2\cap H_1}\frac{ w^- (x) w^-(y)}{|r_{1,2}(x)-r_1(y)|^{N+2s}}\ dxdy\notag\\
&\quad -2\int_{H_2\cap H_1}\int_{H_2\cap H_1}w^+(x)w^-(y)\Big(\frac{1}{|x-y|^{N+2s}}-\frac{1}{|r_2(x)-y|^{N+2s}}\Big)\ dxdy\notag\\
&\quad- 2\int_{H_2\cap H_1}\int_{H_2\cap H_1}w^+(x)w^-(y)\Big(\frac{1}{|x-y|^{N+2s}}-\frac{1}{|r_{1,2}(x)-r_1(y)|^{N+2s}}\Big)\ dxdy\notag\\
&\quad +2\int_{H_2\cap H_1}\int_{H_2\cap H_1}\frac{w^+(x)w^-(y)+w^+(y)w^-(x)}{|r_1(x)-y|^{N+2s}}\ dxdy\notag\\
&\quad -2\int_{H_2\cap H_1}\int_{H_2\cap H_1}\frac{ w^+(x) w^-(y)-w^-(x)w^-(y)}{|r_2(x)-r_1(y)|^{N+2s}}\ dxdy\notag\\
&\quad +2\int_{H_2\cap H_1}\int_{H_2\cap  H_1}\frac{ w^-(x) w^-(y)-w^+(x)w^-(y)}{|r_{1,2}(x)-y|^{N+2s}}\ dxdy\notag\\
&= -4\int_{H_2\cap H_1}\int_{H_2\cap H_1}\frac{w^-(x) w^-(y)}{|r_2(x)-y|^{N+2s}}\ dxdy-4\int_{H_2\cap H_1}\int_{H_2\cap H_1}w^+(x)w^-(y)\Big(\frac{1}{|x-y|^{N+2s}}-\frac{1}{|r_2(x)-y|^{N+2s}}\Big)\ dxdy\notag\\
&\quad +2\int_{H_2\cap H_1}\int_{H_2\cap H_1}\frac{w^+(x)w^-(y)+w^+(y)w^-(x)}{|r_1(x)-y|^{N+2s}}\ dxdy\notag\\
&\quad -4\int_{H_2\cap H_1}\int_{H_2\cap H_1}\frac{  w^+(x) w^-(y) }{|r_{1,2}(x)-y|^{N+2s}}\ dxdy+4\int_{H_2\cap H_1}\int_{H_2\cap H_1}\frac{ w^-(x)w^-(y)}{|r_{1,2}(x)-y|^{N+2s}}\ dxdy\notag\\
&= -4\int_{H_2\cap H_1}\int_{H_2\cap H_1}w^-(x) w^-(y)\Big(\frac{1}{|r_2(x)-y|^{N+2s}}-\frac{ 1}{|r_{1,2}(x)-y|^{N+2s}}\Big)\ dxdy\notag\\
&\quad -4\int_{H_2\cap H_1}\int_{H_2\cap H_1}w^+(x)w^-(y)\Big(\frac{1}{|x-y|^{N+2s}}-\frac{1}{|r_2(x)-y|^{N+2s}}\Big)\ dxdy\notag\\
&\quad +4\int_{H_2\cap H_1}\int_{H_2\cap H_1}\frac{w^+(x)w^-(y)}{|r_1(x)-y|^{N+2s}}\ dxdy-4\int_{H_2\cap H_1}\int_{H_2\cap H_1}\frac{  w^+(x) w^-(y) }{|r_{1,2}(x)-y|^{N+2s}}\ dxdy\notag\\
&= -4\int_{H_2\cap H_1}\int_{H_2\cap H_1}w^-(x) w^-(y)\Big(\frac{1}{|r_2(x)-y|^{N+2s}}-\frac{ 1}{|r_{1,2}(x)-y|^{N+2s}}\Big)\ dxdy\notag\\
&\quad -4\int_{H_2\cap H_1}\int_{H_2\cap H_1}w^+(x)w^-(y)\Big(\frac{1}{|x-y|^{N+2s}}-\frac{1}{|r_2(x)-y|^{N+2s}}-\frac{1}{|r_1(x)-y|^{N+2s}}+\frac{1}{|r_{1,2}(x)-y|^{N+2s}}\Big)\ dxdy.\notag
\end{align} 
From here the statement of the Lemma follows, once we show the following claim:
\begin{equation}
\label{claim-double}
\text{\textit{Claim:}}\quad \frac{1}{|x-y|^{N+2s}}-\frac{1}{|r_2(x)-y|^{N+2s}}-\frac{1}{|r_1(x)-y|^{N+2s}}+\frac{1}{|r_{1,2}(x)-y|^{N+2s}}\geq 0,\quad x,y\in H_1\cap H_2
\end{equation}
We write
\[
\frac{1}{|x-y|^{N+2s}}-\frac{1}{|r_2(x)-y|^{N+2s}}=\frac{1}{|x-y|^{N+2s}}\Bigg(1-\Big(\frac{|x-y|^2}{|r_2(x)-y|^2}\Big)^{\frac{N}{2}+s}\Bigg)
\]
and
\[
\frac{1}{|r_1(x)-y|^{N+2s}}-\frac{1}{|r_{1,2}(x)-y|^{N+2s}}=\frac{1}{|r_1(x)-y|^{N+2s}}\Bigg(1-\Big(\frac{|r_1(x)-y|^2}{|r_{1,2}(x)-y|^2}\Big)^{\frac{N}{2}+s}\Bigg).
\]
In the following, fix $x,y\in H_1\cap H_2$ and without loss we may assume $e_1=(1,0,\ldots,0)$ and $e_2=(0,1,0,\ldots,0)$. Indeed, otherwise we may rotate the half spaces and since $\cE_s$ is invariant under rotations the situation remains the same. Then with $D:=\sum\limits_{k=2}^{N}(x_k-y_k)^2$
\begin{align*}
|r_1(x)-y|^2&=(x_1+y_1)^2+(x_2-y_2)^2+D=4x_1y_1+ (x_1-y_1)^2+(x_2-y_2)^2+D=4x_1y_1+|x-y|^2\\
|r_2(x)-y|^2&=(x_1-y_1)^2+(x_2+y_2)^2+D=4x_2y_2+ (x_1-y_1)^2+(x_2-y_2)^2+D=4x_2y_2+|x-y|^2\\
|r_{1,2}(x)-y|^2&=(x_1+y_1)^2+(x_2+y_2)^2+D=4x_1y_1+4x_2y_2+ (x_1-y_1)^2+(x_2-y_2)^2+D=4x_1y_1+4x_2y_2+|x-y|^2
\end{align*}
Thus with $M:=|x-y|^2$ we have
\begin{align*}
 &\frac{1}{M^{\frac{N}{2}+s}}-\frac{1}{|r_2(x)-y|^{N+2s}}-\frac{1}{|r_1(x)-y|^{N+2s}}+\frac{1}{|r_{1,2}(x)-y|^{N+2s}}\\
&=\frac{1}{M^{\frac{N}{2}+s}}\Bigg(1-\Big(\frac{M}{|r_2(x)-y|^2}\Big)^{\frac{N}{2}+s} -\Big(\frac{M}{|r_1(x)-y|^2}\Big)^{\frac{N}{2}+s}\Bigg(1-\Big(\frac{|r_1(x)-y|^2}{|r_{1}(x)-r_2(y)|^2}\Big)^{\frac{N}{2}+s}\Bigg)\Bigg)\\
&=\frac{1}{M^{\frac{N}{2}+s}}\Bigg(1-\Big(\frac{M}{|r_2(x)-y|^2}\Big)^{\frac{N}{2}+s} -\Big(\frac{M}{|r_1(x)-y|^2}\Big)^{\frac{N}{2}+s} +\Big(\frac{M}{|r_{1,2}(x)-y|^2}\Big)^{\frac{N}{2}+s} \Bigg)\\
&=\frac{1}{M^{\frac{N}{2}+s}}\Bigg(1+\Big(\frac{M}{4x_1y_1+4x_2y_2+M}\Big)^{\frac{N}{2}+s}-\Big(\frac{M}{4x_1y_1+M}\Big)^{\frac{N}{2}+s} -\Big(\frac{M}{4x_2y_2+M}\Big)^{\frac{N}{2}+s}  \Bigg).
\end{align*}
Using the notation $a=4x_1y_1>0$ and $b=4x_2y_2>0$, we may consider for fixed $M>0$ the map
\[
f:[0,\infty)^2\to\R,\quad (a,b)\mapsto 1+\Big(\frac{M}{a+b+M}\Big)^{\frac{N}{2}+s}-\Big(\frac{M}{a+M}\Big)^{\frac{N}{2}+s} -\Big(\frac{M}{b+M}\Big)^{\frac{N}{2}+s}.
\]
Then \eqref{claim-double} follows once $f\geq0$. Note that
\[
\nabla f(a,b)=-(\frac{N}{2}+s)M^{\frac{N}{2}+s}\left(\begin{array}{c}\Big(\frac{1}{a+b+M}\Big)^{\frac{N}{2}+1+s}-\Big(\frac{1}{a+M}\Big)^{\frac{N}{2}+1+s}\\
\Big(\frac{1}{a+b+M}\Big)^{\frac{N}{2}+1+s}-\Big(\frac{1}{b+M}\Big)^{\frac{N}{2}+1+s}
\end{array}\right).
\]
Clearly, $f$ has a saddle node at $(0,0)$, but note that for any $(c,d)\in [0,\infty)^2$ we have
\[
\nabla f(a,b){c\choose d}>0,
\]
so that $f$ is increasing in any direction $(c,d)$. In particular, since $f(0,0)=0$, it follows that $f(a,b)\geq 0$ for $a,b\geq 0$. Hence \eqref{claim-double} follows, which implies the assertion of the lemma.
\end{proof}

In view of Lemma \ref{properties vs} we may define \textit{doubly antisymmetric supersolutions} as follows. Let $U\subset H_1\cap H_2$ and $c\in L^{\infty}(U)$. Then $w\in \cV^s(U)$ is called a doubly antisymmetric supersolution of 
\begin{equation}\label{weak-problem-doubly}
\left\{
\begin{aligned}
(-\Delta)^sw&\geq  c(x)w &&\text{in $U$,}\\
w&\geq 0 &&\text{in $H_1\cap H_2\setminus U$,}
\end{aligned}\right.
\end{equation}
if $w$ is doubly antisymmetric and satisfies
\[
\cE_s(w,\phi)\geq \int_{U} c(x)w(x)\phi(x)\ dx\quad\text{for all nonnegative $\phi\in \cH^s_0(U)$.}
\]

In the following, for an open set $U\subset H_1\cap H_2$ let

\[
\lambda_{1}^-(U):=\min_{\substack{ u\in \cH^s_0(U\cup r_1(U))\\ u\neq 0\\ u\circ r_1\equiv-u}} \frac{\cE_s(u,u)}{\|u\|_{L^2(U\cup r_1(U))}^2}.
\]
 
We emphasize that $\lambda_{1}^{-}(U)>\lambda_{1}(U\cup r_1(U))$, where $\lambda_1(D)$ denotes the first Dirichlet eigenvalue of $(-\Delta)^s$ in $D$. Since (see e.g. \cite[Lemma 2.1]{JW16})
$$
\sup_{\substack{D\subset \R^N\ \text{open}\\ |D|\leq \delta}} \lambda_1(D)\to \infty \quad\text{ as $\delta\to 0$,}
$$
 it follows also that 
\begin{equation}\label{lambda- is large}
\sup_{\substack{U\subset \R^N\ \text{open}\\ |U|\leq \delta}} \lambda_{1}^{-}(U)\to \infty  \quad\text{as $|U|\to 0$.}
\end{equation}
We thus can show the following version of a small volume maximum principle for doubly antisymmetric supersolutions.

\begin{proposition}\label{lem:0.2}
Let $c_{\infty}>0$. Then there is $\delta>0$ such that the following is true. For all $U\subset H_1\cap H_2$ open with $|U|\leq \delta$, $c\in L^{\infty}(U)$ with $c\leq c_{\infty}$, and all doubly antisymmetric supersolutions $w$ of \eqref{weak-problem-doubly} it follows that $w\geq0$ in $H_1\cap H_2$.
\end{proposition}
\begin{proof}
Let $c_{\infty}>0$. By \eqref{lambda- is large}, we may fix $\delta>0$ such that $c_{\infty}\leq \lambda_{1}^{-}(U)$ for all open sets $U\subset H_1\cap H_2$ with $|U|\leq \delta$. Fix such an open set $U$ and let $c\in L^{\infty}(U)$. Then note that we may reflect $c$ evenly across $\partial H_1$. Then we have for any with respect to $\partial H_1$ antisymmetric function $\varphi\in \cH^s_0(V)$, $V=U\cup r_1(U)$ with $\varphi\geq 0$ in $U$:
\begin{align*}
\cE_s(w,\varphi)=\cE_s(w,1_U\varphi)+\cE_s(w,1_{r_1(U)}\varphi)\geq \int_U c(x)w(x)\varphi(x)+\int_{r_1(U)}c(x)w(x)\varphi(x)\ dx=\int_Vc(x)w(x)\varphi(x)\ dx.
\end{align*}
Here, we have used the antisymmetry of $w$ and $\varphi$ with respect to $\partial H_1$ and Lemma \ref{properties vs} to have $1_U\varphi\in \cH^s_0(U)$, $1_{r_1(U)}\varphi\in \cH^s_0(r_1(U))$, and 
\begin{align*}
\cE_s(w,1_{r_1(U)}\varphi)=\cE_s(w\circ r_1,1_U\varphi\circ r_1)=\cE_s(w,1_U\varphi)\geq \int_Uc(x)w(x)\varphi(x)\ dx=\int_{r_1(U)}c(x)w(x)\varphi(x)\ dx,
\end{align*}
since we extended $c$ evenly across $\partial H_1$. Then $v=w^-1_{H_1}1_{H_2}-w^+1_{H_1^c}1_{H_2}\in \cH^s_0(U\cup r_1(U))$ by Lemma \ref{test-function} and we have by symmetry
\begin{align*}
\cE_s(w,v)&=\int_{U}c(x)w(x)w^-(x)\ dx-\int_{r_1(U)}c(x)w(x)w^+(x)\ dx=-\int_{U}c(x)(w^-(x))^2\ dx-\int_{r_1(U)}c(x)(w^+(x))^2\ dx\\
&\geq- \lambda_{1,s}^-\Bigg(\int_{U} (w^-(x))^2\ dx+\int_{r_1(U)} (w^+(x))^2\ dx\Bigg)=- \lambda_{1,s}^-\|v\|_{L^2(V)}^2\geq -  \cE_s(v,v).
\end{align*}
Hence with Lemma \ref{test-function} we have $0\leq \cE_s(w,v)+\cE_s(v,v)\leq 0$ and this can only be true if $v\equiv 0$.
\end{proof}

In the next statement, we give a Hopf type lemma for equation \eqref{weak-problem-doubly} similar to \cite[Proposition 3.3]{FJ15}.

\begin{proposition}
\label{Hopf-lemma-doubly-anti}
Let $U\subset H_1\cap H_2$ open. Furthermore, let $c\in L^\infty(U)$ and let $u\in \cV^s(U)$ be a doubly antisymmetric supersolution of \eqref{weak-problem-doubly}. Assume $u\geq0$ in $H_1\cap H_2$. Then either $u\equiv 0$ or $u>0$ in $U$ in the sense that
\[
\essinf_{K}u>0\quad\text{for all compact sets $K\subset U$.}
\]
Moreover, if there is $x_0\in \partial U\setminus [\partial H_1\cup \partial H_2]$ such that
\begin{enumerate}
\item there exists a ball $B\subset U$ with $\partial B\cap \partial U=\{x_0\}$ and $\lambda_{1,s}^-(B)\geq c$ and
\item $u(x_0)=0$,
\end{enumerate}
then there exists $C>0$ such that
$$
u\geq C\delta^s_{B}\qquad\text{in}\qquad B,
$$
where $\delta_B$ denote the distance to boundary of $B$, and, in particular, if $u\in C(B)$, then
$$
\liminf_{t\downarrow 0}\frac{u(x_0-t\nu(x_0))}{t^s}>0.
$$
\end{proposition}
\begin{proof}
Assume $u\not\equiv 0$. Then there exists a set $K\subset H_1\cap H_2$ such that $|K|>0$ and such that
\begin{equation}
\label{choice-epsilon}
\epsilon:=\essinf_K u>0.
\end{equation}
Let $B\subset U$ be an open ball such that $\dist(B,K)>0$ and $\partial B\cap \partial H_i=\emptyset$ for $i=1,2$. By making $B$ smaller if necessary, we may assume 
\begin{equation}
\label{choice-B}
\lambda_{1,s}^-(B)\geq c
\end{equation}
Let $\psi_{B}\in \cH^s_0(B)$ be the solution to 
$$
\Ds \psi_{B} = 1\qquad\text{in}\qquad B
$$
Recall that there exists $c_i=c_i(N,s,B)>0$, $i=1,2$ such that $c_1\delta^s_B\leq \psi_B\leq c_2\delta^s_B$. For any $\alpha>0$, we define 
$$
\overline u :=\psi_{B} +\alpha 1_{K}-\psi_{r_1(B)}-\alpha  1_{r_1(K)}\qquad\text{and}\qquad w := \overline u-\overline u\circ r_2
$$
It is clear that $w\circ r_1 = -w=w\circ r_2$, that is, $w$ is doubly antisymmetric. Let $\varphi\in\cH^s_0(B)$ with $\varphi\geq 0$. Then, we have 
\begin{align}
&\cE_s(w,\varphi)=\cE_s(\overline u,\varphi)-\cE_s(\overline u\circ r_2,\varphi)\notag\\
&= \int_{B}\varphi(x)\,dx-\alpha c_{N,s}\int_{B}\varphi(x)\int_{K}\frac{dy}{|x-y|^{N+2s}}dx+\alpha c_{N,s}\int_{B}\varphi(x)\int_{r_1(K )}\frac{dy}{|x-y|^{N+2s}}dx\nonumber\\
&\quad+c_{N,s}\int_{B}\varphi(x)\int_{r_1(B)}\frac{\psi_B(y)}{|x-y|^{N+2s}}dy\,dx+c_{N,s}\int_{B\times B}\frac{\psi_B(x)\varphi(y)}{|x-r_2(y)|^{N+2s}}dx\,dy+\alpha c_{N,s}\int_{K}\int_{B}\frac{\varphi(y)}{|x-r_2(y)|^{N+2s}}dy\,dx\nonumber\\
&\quad-\alpha c_{N,s}\int_{r_1(K)\times r_2(B)}\frac{\varphi(r_2(y))}{|x-y|^{N+2s}}dx\,dy-c_{N,s}\int_{r_1(B)\times r_2(B)}\frac{\psi_{r_1(B)}(x)\varphi(r_2(y))}{|x-y|^{N+2s}}dx\,dy\nonumber\\
&=\int_{B}\varphi(x)\Big( 1-\alpha c_{N,s}\int_{K}\Big[\frac{1}{|x-y|^{N+2s}}-\frac{1}{|x-r_1(y)|^{N+2s}}-\frac{1}{|x-r_2(y)|^{N+2s}}+\frac{1}{|x-r_{1,2}(y)|^{N+2s}}\Big]dy\nonumber\\
&\quad+c_{N,s}\int_{B}\frac{\psi_B(r_1(y))}{|x-r_1(y)|^{N+2s}}dy+ c_{N,s}\int_{B}\frac{\psi_B(y)}{|x-r_2(y)|^{N+2s}}dy-c_{N,s}\int_{B}\frac{\psi_{r_1(B)}(r_1(y))}{|x-r_{1,2}(y)|^{N+2s}}dy\Big)\nonumber\\
&\leq \int_{B}\varphi(x)\Big( 1-\alpha c_{N,s}\int_{K}\Big[\frac{1}{|x-y|^{N+2s}}-\frac{1}{|x-r_1(y)|^{N+2s}}-\frac{1}{|x-r_2(y)|^{N+2s}}+\frac{1}{|x-r_{1,2}(y)|^{N+2s}}\Big]dy\nonumber\\
&\quad+c_{N,s}\|\psi_B\|_{L^\infty(\R^N)}\int_{B}\Big[\frac{1}{|x-r_1(y)|^{N+2s}}+ \frac{1}{|x-r_2(y)|^{N+2s}}+\frac{1}{|x-r_{1,2}(y)|^{N+2s}}\Big]dy\Big)\nonumber\\
&\leq  \int_{B}\varphi(x)\Big( \kappa-\alpha c_{N,s}\int_{K}\Big[\frac{1}{|x-y|^{N+2s}}-\frac{1}{|x-r_1(y)|^{N+2s}}-\frac{1}{|x-r_2(y)|^{N+2s}}+\frac{1}{|x-r_{1,2}(y)|^{N+2s}}\Big]dy\Big),\label{bilinear-form}
\end{align}
with 
$$
\kappa := 1+c_{N,s}\|\psi_B\|_{L^\infty(\R^N)}\int_{B}\Big[\frac{1}{|x-r_1(y)|^{N+2s}}+ \frac{1}{|x-r_2(y)|^{N+2s}}+\frac{1}{|x-r_{1,2}(y)|^{N+2s}}\Big]dy<\infty,
$$
where we have used that the boundary of $B$ does not touch $\partial H_1\cup \partial H_2$. Since $\overline{B}$ and $K$ are compactly contained in $H_1\cap H_2$, it follows that
$$
C:=\inf_{x\in B,\ y\in K} \Bigg(\frac{1}{|x-y|^{N+2s}}-\frac{1}{|x-r_1(y)|^{N+2s}}- \frac{1}{|x-r_2(y)|^{N+2s}}+\frac{1}{|x-r_{1,2}(y)|^{N+2s}}\Bigg)>0.
$$
Since $c, \psi_B\in L^\infty(U)$, we may hence choose $\alpha$ large enough so that
$$
\kappa-\alpha c_{N,s}\int_{K}\Big[\frac{1}{|x-y|^{N+2s}}-\frac{1}{|x-r_1(y)|^{N+2s}}-\frac{1}{|x-r_2(y)|^{N+2s}}+\frac{1}{|x-r_{1,2}(y)|{N+2s}}\Big]dy\leq c(x)\psi_B(x)\quad\forall\,x\in B.
$$
Consequently, equation \eqref{bilinear-form} gives
$$
\cE_s(w,\varphi)\leq \int_{B}c(x)\varphi(x)\psi_B(x)\,dx\qquad \text{for all nonnegative $\varphi\in \cH^s_0(B)$.}
$$
Therefore $-w$ satisfies in weak sense
\begin{equation}
\label{problem in B}
\left\{
\begin{aligned}
(-\Delta)^s(-w)&\geq c(x)(-w) &&\text{in $B$,}\\
(-w)&\geq 0 &&\text{in $H_1\cap H_2\setminus B$,}\\
-w\circ r_i&=w&&\text{in $\R^N$ for $i=1,2$.}
\end{aligned}\right.
\end{equation}
Next, consider $u_\epsilon:=u-\frac{\epsilon}{\alpha} w$ with $\epsilon$ given in \eqref{choice-epsilon}. Then $u_{\epsilon}$ also satisfies in weak sense \eqref{problem in B} where the \textit{nonlocal boundary condition} is satisfied by the choice of $\epsilon$. By \eqref{choice-B} and Lemma \ref{lem:0.2} we conclude that $u\geq \frac{\epsilon}{\alpha} \psi_B\geq \frac{\epsilon}{\alpha} c_1\delta^s_B$ in $B$. Since $B$ is chosen arbitrary, the above implies that $u>0$ in $U$ as stated.
If in addition there is $x_0\in \partial U\setminus[\partial H_1\cup \partial H_2]$ with the given properties, the above argument yields in particular
$$
\liminf_{t\downarrow 0}\frac{u(x_0-t\nu(x_0))}{t^s}\geq \epsilon \lim_{t\downarrow 0}\frac{\psi_B(x_0-t\nu(x_0))}{t^s}>0.
$$
This finishes the proof.
\end{proof}

\begin{remark}
To put the Hopf type statement in Proposition \ref{Hopf-lemma-doubly-anti} into perspective, consider in Problem \eqref{main-prob} the nonlinearity $f(x,u)=|u|^{2^{\ast}_s-2}u$ with $2^{\ast}_s:=\frac{2N}{N-2s}$, the critical fractional exponent. It was shown in \cite{RS12-2} that there is no positive bounded solution if $\Omega$ is starshaped. Up to our knowledge, it remains an open question, if there is a \textit{sign-changing} solution to this problem. Assuming that $\Omega$ is bounded and starshaped with $C^{1,1}$ boundary and there exists a bounded solution of \eqref{main-prob} with $f(x,u)=|u|^{2^{\ast}_s-2}u$, it first follows that $u\in C^s(\R^N)\cap C^{\infty}(\Omega)$ (see e.g. \cite{RS14}) and the fractional Pohozaev identity from \cite{RS12-2} implies
\[
\int_{\partial\Omega}\Big(\frac{u}{\dist(\cdot,\partial\Omega)^s}\Big)^2(x\cdot \nu)\ d\sigma=0.
\]
However, by \cite[Proposition 3.3]{FJ15} it then follows that if $\Omega$ has additionally a symmetry hyperplane $T$ and $u$ is odd with respect to reflections across this hyperplane and of one sign on one side of the hyperplane, then $\Big(\frac{u}{\dist(\cdot,\partial\Omega)^s}\Big)^2>0$ on $\partial \Omega\setminus T$. Whence, there cannot be such an odd solution of the problem. Similarly, using instead Proposition \ref{Hopf-lemma-doubly-anti}, it follows that there can also be no doubly antisymmetric solution of this problem if $\Omega$ satisfies (D). 
\end{remark}

\section{Symmetry of solutions} \label{symmetry}

In the following, we use the notation from Section \ref{linear} and assume $\Omega\subset \R^N$ satisfies (D). Moreover, $f\in C(\Omega\times \R)$ satisfies (F1) and (F2) and let $u\in L^{\infty}(U)\cap \cH^s_0(\Omega)$ be a solution of problem \eqref{main-prob} which satisfies $u\circ r_{N,0}=-u$. Note that by (F1) and \cite{RS14} it follows that $u\in C^s(\R^N)$. For $\lambda\in \R$ we may than define
$$
v_{\lambda}(x)=u(r_{\lambda,1}(x))-u(x).
$$
Then it follows that $v_{\lambda}$ is antisymmetric with respect to $H_{N,0}$ and $H_{1,\lambda}$, hence doubly antisymmetric, and it satisfies due to (F2)
\begin{equation}\label{linearization}
\left\{\begin{aligned}
(-\Delta)^sv_{\lambda}&\geq c_{\lambda}(x)v_{\lambda}&& \text{in $\Omega_{\lambda}:=\Omega\cap H_{N,0}\cap H_{1,\lambda}$,}\\
v_{\lambda} &\geq 0 && \text{in $H_{N,0}\cap H_{1,\lambda}\setminus \Omega_{\lambda}$,}
\end{aligned}\right.
\end{equation}
where
\[
c_{\lambda}(x)=\left\{\begin{aligned} &\frac{f(x,u(r_{\lambda,1}(x)))-f(x,u(x))}{u(r_{\lambda,1}(x))-u(x)} && u(r_{\lambda,1}(x))\neq u(x)\\
&0 && u(r_{\lambda,1}(x))=u(x)
\end{aligned}\right.
\]
Note that by assumption (F1) we have
\[
\sup_{\lambda\in \R}\sup_{x\in \Omega_{\lambda}}|c_{\lambda}(x)|=:c_{\infty}<\infty.
\]
Finally, let $\lambda_1:=\sup_{x\in \Omega} x_1$.

\begin{proof}[Proof of Theorem \ref{main-thm1}]
Assume that $u$ is nontrivial. We apply the moving plane method to then prove that $u$ is symmetric with respect to $H_{1,0}$ and decreasing in $x_1$. For this let 
$$
\lambda_0:=\inf\{\lambda\in(0,\lambda_1)\;:\; v_{\mu}> 0 \, \text{ in $\Omega_{\mu}$ for all $\mu\in(\lambda,\lambda_1)$}\}
$$
Next note that by (D) and Proposition \ref{lem:0.2} it follows that there is $\epsilon>0$ such that $v_{\mu}\geq 0$ for all $\lambda\in (\lambda_1-\epsilon,\lambda_1)$ and thus by Proposition \ref{Hopf-lemma-doubly-anti} we have $\lambda_0\leq \lambda_1-\epsilon$. Assume next by contradiction that $\lambda_0>0$. Then by continuity $v_{\lambda_0}\geq 0$ in $H_{N,0}\cap H_{1,\lambda_0}$. By Proposition \ref{Hopf-lemma-doubly-anti} it follows that either $v_{\lambda_0}\equiv 0$ or $v_{\lambda_0}>0$.\\
If $v_{\lambda_0}\equiv 0$, this implies that we have $u\equiv 0$ in $\Omega \setminus H_{1,\lambda_0-\lambda_1}$. But then, we can also start moving the hyperplane from the left (working instead with $\R^N\setminus H_{1,\lambda}$), up to the same $\lambda_0$. It then follows that $u$ has two different parallel symmetry hyperplanes, but since $u$ vanishes outside of $\Omega$, this implies $u\equiv 0$, which cannot be the case.\\
If $v_{\lambda_0}>0$, let $\delta>0$ be given by Proposition \ref{lem:0.2} according to $c_{\infty}$. Then by continuity there is $\mu>0$ such that and a compact set $K\subset \Omega_{\lambda_0}$ such that $|\Omega_{\lambda_0}\setminus K|\leq \frac{\delta}{2}$ and $v_{\lambda_0}\geq 2\mu$ in $K$. Again, by continuity, we can find $\tau\in(0,\lambda_1-\lambda_0)$ such that $v_{\lambda}\geq \mu$ for all $\lambda\in[\lambda_0-\tau,\lambda_0]$. Let $U_{\lambda}:=\{x\in \Omega_{\lambda}\;:\; v_{\lambda}<0\}$. Then, by making $\tau$ smaller if necessary, we may also assume $|U_{\lambda}|\leq \delta$ for all $\lambda\in[\lambda_0-\tau,\lambda_0]$. A combination of Proposition \ref{lem:0.2} and \ref{Hopf-lemma-doubly-anti} gives a contradiction to the definition of $\lambda_0$.\\
Whence, $\lambda_0>0$ is not possible. Thus $\lambda_0=0$ and this finishes the proof.
\end{proof}

\section{A symmetric sign-changing solution} \label{application} 

Let $\Omega\subset \R^N$ open and bounded and consider the functional
$$
J:\cH^s_0(\Omega)\to \R, \qquad J(u)=\cE_s(u,u).
$$
Let $M:=\{u\in \cH^s_0(\Omega)\;:\; u=-u\circ r_N,\ \int_{\Omega}|u(x)|^p\ dx=1\}$ with $1<p<\frac{2N}{N-2s}$. Then by a constraint minimization argument using the framework as explained e.g. in \cite{SV12,SV13}, see also \cite{BP16}, it follows that there exists such a minimizer $u$ of $J|_{M}$. That is, the minimum
\begin{equation}\label{defi-lambda}
\lambda_{1,p}^-=\min_{u\in M}\cE_s(u,u)
\end{equation}
is attained. Similar to \cite[Theorem 3.1]{BP16}, it can be shown that this minimizer is bounded and then, by an iteration of the results of \cite{RS14,G15:2}, we have $u\in C^{\infty}(\Omega)$. If in addition $\partial \Omega$ is of class $C^{1,1}$, then \cite{RS14,G15:2} also imply that $u\in C^s(\R^N)$.

\begin{proof}[Proof of Theorem \ref{main-thm2}]
Let $\lambda^-_{1,p}$ be as in \eqref{defi-lambda} 
and let $u$ be the minimizer as explained in the above remarks. In view of Theorem \ref{main-thm1} it remains to show that $u$ can be chosen of one sign in $\O^+:=\Omega\cap H_{N,0}$. In the following $\O^-=\O\setminus \O^+$. Assume by contradiction that $u$ changes sign in $\O^+$ and let $\O^+_1:=\{x\in \O^+: u(x)>0\}$ and $\O^+_2:=\{x\in \O^+: u(x)\leq 0\}$. We also let $\O^-_1 = r_{N,0}(\O^+_1)$, and $\O^-_2 = r_{N,0}(\O^+_2)$. By the property of $u$, it is clear that $u<0$ in $\O^-_1$ and $u\geq 0$ in $\O^-_2$. Now let $\overline u$ be defined by 
\be
\overline u = 1_{\O^+}|u|-1_{\O^-}|u|.
\ee

Then $\ov u\in M$, that is $\ov u\in \cH^s_0(\O)$ satisfies $\ov u\circ r_{N,0}=-\ov u$ and
\be\label{eq:l2-norm-ov u}
\int_{\O} |\ov u|^pdx = \int_{\O}(|\ov u|^2)^{p/2}dx=\int_{\O}(1_{\O^+} |u|^2+1_{\O^-}|u|^2)^{p/2}dx= \int_{\O}|u|^pdx=1.
\ee 
 Moreover, we have 
\begin{align}\label{eq:semi-norm-ov u}
&\frac{2}{C_{N,s}}\cE_s(\ov u,\ov u)=\int_{\R^N\times\R^N} \frac{\big(\ov u(x)-\ov u(y)\big)^2}{|x-y|^{N+2s}}dxdy=\int_{\O\times\O} \frac{\big(\ov u(x)-\ov u(y)\big)^2}{|x-y|^{N+2s}}dxdy+2\int_{\O}\ov u^2(x)\int_{\R^N\setminus\O}\frac{dy}{|x-y|^{N+2s}}\,dx\nonumber\\
&=\int_{\O^+\times\O} \frac{\big(\ov u(x)-\ov u(y)\big)^2}{|x-y|^{N+2s}}dxdy+\int_{\O^-\times\O} \frac{\big(\ov u(x)-\ov u(y)\big)^2}{|x-y|^{N+2s}}dxdy+4\int_{\O^+}\int_{\R^N\setminus\O}\frac{ u^2(x)dy}{|x-y|^{N+2s}}\,dx\nonumber\\
&=\int_{\O^+\times\O^+} \frac{\big( |u(x)|- |u(y)|\big)^2}{|x-y|^{N+2s}}dxdy+\int_{\O^-\times\O^-} \frac{\big( |u(x)|- |u(y)|\big)^2}{|x-y|^{N+2s}}dxdy+2\int_{\O^-\times\O^+} \frac{\big( |u|(x)+ |u|(y)\big)^2}{|x-y|^{N+2s}}dxdy\nonumber\\
&\qquad\qquad +4\int_{\O^+} u^2(x)\int_{\R^N\setminus\O}\frac{dy}{|x-y|^{N+2s}}\,dx.
\end{align}
Using the notation above, we rewrite
\begin{align}\label{eq:int-omega+}
\int_{\O^+\times\O^+} \frac{\big( |u(x)|- |u(y)|\big)^2}{|x-y|^{N+2s}}dxdy &= \int_{\O^+\times\O^+} \frac{\big( u(x)- u(y)\big)^2}{|x-y|^{N+2s}}dxdy +2\int_{\O^+_1\times\O^+_2}\frac{(u(x)+u(y))^2-(u(x)-u(y))^2}{|x-y|^{N+2s}}dxdy\nonumber\\
&=\int_{\O^+\times\O^+} \frac{\big( u(x)- u(y)\big)^2}{|x-y|^{N+2s}}dxdy+4\int_{\O^+_1\times\O^+_2}\frac{u(x)u(y)}{|x-y|^{N+2s}}dxdy.
\end{align}
Similarly we have 
\begin{align}\label{eq:int-omega-}
\int_{\O^-\times\O^-} \frac{\big( |u(x)|- |u(y)|\big)^2}{|x-y|^{N+2s}}dxdy &= \int_{\O^-\times\O^-} \frac{\big( u(x)- u(y)\big)^2}{|x-y|^{N+2s}}dxdy+4\int_{\O^-_1\times\O^-_2}\frac{u(x)u(y)}{|x-y|^{N+2s}}dxdy\nonumber\\
& = \int_{\O^-\times\O^-} \frac{\big( u(x)- u(y)\big)^2}{|x-y|^{N+2s}}dxdy+4\int_{\O^+_1\times\O^+_2}\frac{u(x)u(y)}{|x-y|^{N+2s}}dxdy.
\end{align}
Now using that $\O^-_j = r_N(\O^+_j)$, $j=1,2$ we get 
\begin{align}\label{eq:int-omega+-}
\int_{\O^-\times\O^+} &\frac{\big(  |u(x)|-|u(y)|\big)^2}{|x-y|^{N+2s}}dxdy = \int_{\O^-\times\O^+} \frac{\big( u(x)- u(y)\big)^2}{|x-y|^{N+2s}}dxdy+\int_{\O^-_2\times\O^+_1}\frac{(u(x)+u(y))^2-(u(x)-u(y))^2}{|x-y|^{N+2s}}dxdy\nonumber\\
&\qquad\qquad\qquad\qquad\qquad\qquad\qquad+\int_{\O^-_1\times\O^+_2}\frac{(u(x)+u(y))^2-(u(x)-u(y))^2}{|x-y|^{N+2s}}dxdy\nonumber\\
&=\int_{\O^-\times\O^+} \frac{\big( u(x)- u(y)\big)^2}{|x-y|^{N+2s}}dxdy+2\int_{\O^-_2\times\O^+_1}\frac{u(x)u(y)}{|x-y|^{N+2s}}dxdy+2\int_{\O^-_1\times\O^+_2}\frac{u(x)u(y)}{|x-y|^{N+2s}}dxdy\nonumber\\
&=\int_{\O^-\times\O^+} \frac{\big( u(x)- u(y)\big)^2}{|x-y|^{N+2s}}dxdy -4\int_{\O^+_1\times\O^+_2}\frac{u(x)u(y)}{|r_{N,0}(x)-y|^{N+2s}}dxdy.
\end{align}
Summing up \eqref{eq:int-omega+}, \eqref{eq:int-omega-} and \eqref{eq:int-omega+-}, and taking into account \eqref{eq:semi-norm-ov u}, we obtain
\begin{align}\label{eq:semi-norm- ov u-2}
&\frac{2}{C_{N,s}}\cE_s(\ov u,\ov u)-4\int_{\O^+} u^2(x)\int_{\R^N\setminus\O}\frac{dy}{|x-y|^{N+2s}}\,dx = \int_{\R^N\times\R^N}\frac{(u(x)-u(y))^2}{|x-y|^{N+2s}}dxdy-2\int_{\O}u^2(x)\int_{\R^N\setminus\O}\frac{dy}{|x-y|^{N+2s}}\nonumber\\
&+8\int_{\O^+_1\times\O^+_2}\frac{u(x)u(y)}{|x-y|^{N+2s}}dxdy-8\int_{\O^+_2\times\O^+_1}\frac{u(x)u(y)}{|r_{N,0}(x)-y|^{N+2s}}dxdy.
\end{align}
By a change of variable it is clear that 
\[2\int_{\O}u^2(x)\int_{\R^N\setminus\O}\frac{dy}{|x-y|^{N+2s}} dx= 4\int_{\O^+} u^2(x)\int_{\R^N\setminus\O}\frac{dy}{|x-y|^{N+2s}}\,dx.\]
Putting that into \eqref{eq:semi-norm- ov u-2} gives
\begin{align}\label{eq:semi-norm ov u-3}
\frac{2}{C_{N,s}}\cE_s(\ov u,\ov u) &=  \int_{\R^N\times\R^N}\frac{(u(x)-u(y))^2}{|x-y|^{N+2s}}dxdy+8\int_{\O^+_1\times\O^+_2}u(x)u(y)\Big[\frac{1}{|x-y|^{N+2s}}-\frac{1}{|r_{N,0}(x)-y|^{N+2s}}\Big]dxdy.
\end{align}
Now since $\ov u\circ r_N = -\ov u$, it follows from the variational characterization of $\lambda^-_{1,p}(\O)$ in \eqref{defi-lambda} with \eqref{eq:semi-norm ov u-3} and \eqref{eq:l2-norm-ov u} that 
\begin{align*}
\lambda^-_{1,p}(\O)&\leq \cE_s(\ov u,\ov u)= \cE_s( u, u)+4C_{N,s}\int_{\O^+_1\times\O^+_2}u(x)u(y)\Big[\frac{1}{|x-y|^{N+2s}}-\frac{1}{|r_{N,0}(x)-y|^{N+2s}}\Big]dxdy\\
&=\lambda^-_{1,p}(\O)+4C_{N,s}\int_{\O^+_1\times\O^+_2}u(x)u(y)\Big[\frac{1}{|x-y|^{N+2s}}-\frac{1}{|r_{N,0}(x)-y|^{N+2s}}\Big]dxdy.
\end{align*}
That is 
\[
0\leq\int_{\O^+_1\times\O^+_2}u(x)u(y)\Big[\frac{1}{|x-y|^{N+2s}}-\frac{1}{|r_{N,0}(x)-y|^{N+2s}}\Big]dxdy\leq 0.
\]
Whence $u \equiv 0$ in $\O^+_2$ and therefore $u\geq 0$ in $\O^+$. This is in contradiction with the hypothesis. It follows that $u$ does not change sign in $\O^+$ and, without loss of generality, we may assume $u\geq 0$ in $\O^+$. By the strong maximum principle \cite[Corollary $3.4$]{FJ15} we have $u>0$ in $\O^+$.\\
For the additional statement let $p=2$ and let $u, v$ be two normalized minimizers for $\lambda^-_{1,2}(\Omega)$. Assume further they satisfy the sign property in Theorem \ref{main-thm2}, i.e. they are of one sign in $\Omega\cap H_{N,0}$. Then, if $u-v$ is not identically zero, it must change sign in $\Omega\cap H_{N,0}$. Indeed, if not, we may assume  $u-v>0$ in $\O\cap H_{N,0}$ by \cite[Corollary $3.4$]{FJ15}. Therefore $1=\int_{\O}u^2dx = 2\int_{\O\cap H_{N,0}}u^2dx>2\int_{\O\cap H_{N,0}}v^2dx=1$ a contradiction. Note that if $u\not\equiv v$, then also $(u-v)/\|u-v\|_{L^2(\Omega)}$ is a minimizer. But by the above argument, $u-v$ cannot change sign in $\O\cap H_{N,0}$. Whence $u\equiv v$ as claimed.
\end{proof}

\textbf{Acknowledgements.} This work is supported by DAAD and BMBF (Germany) within the project 57385104. The authors would like to thank Mouhamed Moustapha Fall and Tobias Weth for helpful discussions.

\bibliographystyle{amsplain}

\end{document}